\documentclass[a4paper, 12pt]{amsart}
\usepackage{hyperref,xy,doi}
\usepackage{diagbox}
\usepackage{soul}
\usepackage{enumerate}
\usepackage[margin= 0.9 in]{geometry}
\xyoption{all}
\usepackage{amsmath,amssymb, xy,wasysym,bm,amsthm}
\usepackage{tikz}
\usepackage{youngtab, ytableau}
\usepackage{subcaption}
\usetikzlibrary{cd}
\usepackage{color, comment}
\usepackage{multirow}
\theoremstyle{plain}
\newtheorem{theorem}{Theorem}[section]
\newtheorem{corollary}[theorem]{Corollary}
\newtheorem{lemma}[theorem]{Lemma}
\newtheorem{definition}[theorem]{Definition}
\newtheorem{remark}[theorem]{Remark}
\newtheorem{prop}[theorem]{Proposition}

\newtheorem{example}[theorem]{Example}
\numberwithin{equation}{section}

\definecolor{darkred}{rgb}{0.7,0,0} 

\newcommand{\Z}{{\mathbb Z}}

\newcommand{\G}{\text{G}}

\newcommand{\Zr}{\mathbb{Z}_r}

\newcommand{\Grqn}{G(r,q,n)}
\newcommand{\Y}{{\mathcal Y}}
\newcommand{\La}{\lambda}
\newcommand{\bla}{\bm{\lambda}}
\newcommand{\bnu}{\bm{\nu}}
\newcommand{\tla}{\tilde{\bla}}
\newcommand{\tbla}{\tilde{\bla}}

\title[On Quasi Steinberg characters of Complex Reflection Groups]{On Quasi Steinberg characters of Complex \\ Reflection Groups}
\author[Mishra]{Ashish Mishra}
\address{Faculdade de Matemática, Universidade Federal do Pará, Brazil.}
\email{ashishmsr84@gmail.com}
\author[Paul]{Digjoy Paul}
\address{School of Mathematics, Tata Institute of Fundamental Research, Mumbai, India.}
\email{digjoypaul@gmail.com}
\author[Singla]{Pooja Singla}
\address{ Department of Mathematics and Statistics, Indian Institute of Technology Kanpur, Kanpur 208016, India. }
\email{psingla@iitk.ac.in}
\begin{document}
	
    \begin{abstract}
        Let $G$ be a finite group and $p$ be a prime number dividing the order of $G$. An irreducible character $\chi$ of $G$  is called a quasi $p$-Steinberg character if $\chi(g)$ is nonzero for every $p$-regular element $g$ in $G$. In this paper, we classify quasi $p$-Steinberg characters of the complex reflection groups $G(r,q,n)$. In particular, we obtain this classification for Weyl groups of type $B_n$ and type $D_n$. 
    \end{abstract}
\subjclass[2020]{05E10, 20F55, 20C15}
\keywords{Quasi Steinberg characters, Complex reflection groups, Weyl groups, Murnaghan--Nakayama rule}
\maketitle
	\section{Introduction}\label{intro}
	
	The  Steinberg characters and their importance are well known for the finite groups of Lie type; see Steinberg~\cite{MR80669, MR87659}, Curtis~\cite{MR201524}, and  Humphreys~\cite{MR876960}. 
Using the intrinsic property of the Steinberg character, Feit~\cite{F93} defined $p$-Steinberg character for any finite group $G$ and a prime $p$ dividing the order of $G$ (denoted $|G|$). Recall that an irreducible character $\theta$ of $G$ is called a \emph{$p$-Steinberg character} of $G$ if $\theta(x) = \pm |C_G(x)|_p$
	for every  $p$-regular element $x$ in $G$, i.e., $x$  having order co-prime to $p$.  Here $C_G(x)$ denotes the centralizer of $x$ in $G$ and $|n|_p$ denotes the $p$-part of an integer $n$.  Feit studied various properties of these $p$-Steinberg characters. He also conjectured that if a finite simple group $G$ has a $p$-Steinberg character, then $G$ is isomorphic to a simple group of Lie type in characteristic $p$. Darafasheh~\cite{D96}  and Tiep  ~\cite{T97} proved this conjecture. 
Recently,  many variants of $p$-Steinberg characters have been studied; see for example Pellegrini--Zalesski \cite{PZ16} and Malle--Zalesski \cite{MZ20}. These all have been  shown to be of importance in understanding the structure of finite groups. To answer a question of Dipendra Prasad, two of the authors of this article introduced the following notion of quasi $p$-Steinberg character of a finite group $G$ in \cite{PS22}.  
\begin{definition} 
Let  $G$ be a finite group and $p$ be a prime dividing the order of $G$. An irreducible character $\chi$ of $G$  is called a \emph{quasi $p$-Steinberg} character if $\chi(g) \neq 0$ for every $p$-regular element $g$ in $G$.
\end{definition} 

It is easy to see that every $p$-Steinberg character of a finite group $G$ is a quasi $p$-Steinberg character, but the converse is not true. The notion of quasi $p$-Steinberg also differs from other known variants of $p$-Steinberg character, see~\cite[Remark~1.11]{PS22}. In \cite{PS22},  authors gave a classification of all quasi $p$-Steinberg characters of Symmetric groups (see Table \ref{tab:quasi-Sn}), Alternating groups and their double covers. It is natural to ask for the classification of quasi $p$-Steinberg characters of any finite group $G$. Note that every linear character (one-dimensional) of $G$ is a quasi $p$-Steinberg character of $G$ for $p \mid |G|$. Therefore, we will only focus on non-linear characters of $G$.  

In this work, we classify the quasi $p$-Steinberg characters of the infinite family of finite irreducible complex reflection groups denoted by $G(r,q,n)$, see Section~\ref{sec: preliminaries} for the definition and other related results regarding these groups. In particular, we also classify quasi $p$-Steinberg characters of Weyl groups of type $B_n$ and type $D_n$. As mentioned above, the parallel classification for Weyl groups of type $A_n$ was obtained in \cite{PS22}. 

We now describe the main results of this paper. Our first result is a general result that is true for all finite groups. 

\begin{theorem}\phantomsection\label{p23}
	\begin{enumerate}[(i)]
	\item For $n \in \{2,3,4\}$, any $n$-dimensional irreducible character of a finite group $G$ is a quasi $p$-Steinberg character of $G$ for $p \mid n$.
	\item For a finite group $G$ and an automorphism $\mathfrak{a}$ of $G$, an irreducible character $\chi$ of $G$ is quasi $p$-Steinberg if and only if the irreducible character $^\mathfrak{a}\chi$, defined by $^\mathfrak{a}\chi(g) = \chi(\mathfrak{a}(g))$ for $g \in G$, is quasi $p$-Steinberg. 
	\end{enumerate} 
\end{theorem}

The following useful corollary follows directly from Theorem \ref{p23}(ii) and Clifford theory.

\begin{corollary}\label{resqs}
	For a normal subgroup $N$ and an irreducible character $\chi$ of $G$, either all the irreducible characters of $N$ appearing in the restriction $\chi|_N$ are quasi $p$-Steinberg characters or none of these are. 
\end{corollary}

A proof of Theorem \ref{p23} is included in Section~\ref{sec:main-results}. We observe that Theorem \ref{p23}(i) does not hold true for any $n \geq 5$, see Remark~\ref{rk:p23}. In view of Theorem \ref{p23}(i), we call an irreducible character $\chi$ of $G$ to be a {\it hefty} character if $\chi$ has dimension greater than or equal to $5$. While proving our results, we will focus on characterizing the hefty quasi $p$-Steinberg characters of $G(r, q, n)$.

Let $\Y(r,n)$ denote the \emph{set of all $r$-partite Young diagrams} with total number of boxes being $n$. It is well-known that the set $\Y(r,n)$ indexes the irreducible representations of $G(r,1,n)$. For $\bla \in \Y(r,n)$, we use $\chi^{\bla}$ to denote the corresponding irreducible character of $G(r, 1, n)$. In the following definition, we consider certain special elements of $\Y(r,n)$. 

	\begin{definition}\phantomsection\label{lajk}
 For $j\in \{0,1,\cdots r-1\}$, define $\widehat{\bla}^j= (\bla_0, \bla_1, \ldots, \bla_{r-1}) \in \Y(r,n)$ by
$$
 	\bla_j \vdash n \text{ and } \bla_l = \emptyset \text{ for } l \neq j,$$
 	and for $j,k \in \{0,1,\cdots r-1\}$ such that $j \neq k$, define $\widehat{\bla}^{j,k}=(\bla_0, \bla_1, \ldots, \bla_{r-1}) \in \Y(r,n)$ by
 	$$ \bla_j \vdash n-1, \bla_k = (1) \text{ and } \bla_l = \emptyset \text{ for } l \notin \{j,k\}.$$

\end{definition}

The following theorem gives a classification of quasi $p$-Steinberg characters of $G(r,1,n)$ which we prove in Section \ref{sec:main-results}.	

\begin{theorem}\label{MTq1}
	For $n\geq 2$, let $\bla = (\bla_0, \bla_1, \ldots, \bla_{r-1})$ be an $r$-partite partition of $n$ such that $\bla_t \notin \{(n), (1^n)\}$ for any $0 \leq t \leq r-1$. All triples $(n,\bla, p)$ such that $\chi^{\bla}$ is a quasi $p$-Steinberg character of $G(r,1,n)$ are given in Table \ref{tab:quasi-Gr1n}.

	\begin{table}[ht] 
		\begin{tabular}[t]{c|c|c|c}
			\,\,\,\,\, \,\,\, $n$ \,\,\,\,\, \,\,\,& \,\,\,\,\,$\bla$ \, \, &  \,\,\,\,\, \,\,\ $\bla_j$\,\,\, \, \,\,\,&\,\,\,\,\, \,\,\, $p$  \,\,\,\,\, \,\,\, \\
			\hline \hline 
			$2$ & $\widehat{\bla}^{j,k}$ & $(1)$ &$2$ \\
			\hline
			$3$ & $\widehat{\bla}^j$& $(2,1)$ & $2$ \\
			$3$ & $\widehat{\bla}^{j,k}$&$(2), (1,1)$ & $3$ \\
			\hline 
			$4$ & $\widehat{\bla}^j$ & $(2,2)$ & $2$ \\
			$4$ & $\widehat{\bla}^{j,k}$& $(3), (2,1), (1,1,1)$& $2$ \\
			$4$ & $\widehat{\bla}^j$& $(3,1), (2,1,1)$  & $3$ \\
			\hline 
			$5$ &  $\widehat{\bla}^j$ & $(4,1)$, $(2,1,1,1)$ & $2$ \\
			$5$ & $\widehat{\bla}^j$ & $(3,2)$, $(2,2,1)$ & $5$ \\
			\hline 
			$6$ & $\widehat{\bla}^j$ & $(3,2,1)$ & $2$ \\
			$6$ & $\widehat{\bla}^j$ & $(4,2)$, $(2,2,1,1)$ & $3$ \\
			\hline 
			$8$ & $\widehat{\bla}^j$ & $(5,2,1), (3,2,1,1,1)$ & $2$ \\
			\hline 
		\end{tabular}
		\vspace{.2cm} 
		\caption{quasi $p$-Steinberg characters of $G(r,1,n)$}\label{tab:quasi-Gr1n}
	\end{table}
\end{theorem}
Note that while describing $\widehat{\bla}^{j,k}$ in Table \ref{tab:quasi-Gr1n}, we mention only $\bla_j$ because $\bla_k=(1)$. To prove this result, we use Murnaghan--Nakayama rule for $G(r, 1, n)$ along with Theorem~\ref{p23}(i).  Given $\bla = (\bla_0, \bla_1, \ldots, \bla_{r-1}) \in \Y(r,n)$, we use the notation $(\chi^{\bla})^*$ to denote an irreducible character of $\Grqn$ which appears in the decomposition of the restriction of the irreducible character $\chi^{\bla}$ of $G(r,1,n)$ to $\Grqn$. In our next result, we classify all hefty quasi $p$-Steinberg characters of $G(r,q,n)$, see Theorem~\ref{MTqnot1} for a finer statement from which the following result follows.
	\begin{theorem}\label{MTqnot1-hefty}
Let $n \geq 2$ and $(\chi^{\bla})^*$ be a hefty irreducible character  of $G(r,q,n)$ associated with an $r$-partite partition $\bla = (\bla_0, \bla_1, \ldots, \bla_{r-1})$ of $n$. Then $(\chi^{\bla})^*$  is a quasi $p$-Steinberg characters of $G(r, q, n)$ if and only  $\bla$ is as given in the Table \ref{tab:quasi-Gr1n} and $\chi^{\bla}$ is a hefty irreducible character of $G(r, 1, n)$. 
\end{theorem}

We use Clifford theory (see \cite[Chapter~6]{MR0460423}), Theorem~\ref{p23}(i) and Theorem~\ref{MTq1} to prove this result; see Section~\ref{sec:main-results} for the proof. We now list two corollaries of our main results. The following result extends the parallel known results for $G(1, 1, n)$ (see \cite[Corollary~1.7]{PS22}).
\begin{corollary}
	\label{cor:large-n-result}
	For $n \geq 9$ and $p \leq n$, every non-linear irreducible character $\chi$ of $G(r, q, n)$ has a zero at some $p$-regular element of $G(r, q, n)$, i.e. there exists $p$-regular $g\in G(r, q, n)$ such that $\chi(g) = 0$. 
\end{corollary}

The proof of this result follows directly from Theorem~\ref{MTq1} for $q=1$ and from Theorem~\ref{MTqnot1} for $q \neq 1$ . The next result follows directly from Theorem~\ref{p23}(i) and  Corollary~\ref{cor:large-n-result}. 

\begin{corollary}
	\label{cor:large-n-hefty}
	For $n \geq 9$, every non-linear irreducible representation of $G(r, q, n)$ has dimension greater than or equal to five. 
\end{corollary}

\section{Preliminaries}
\label{sec: preliminaries}
\subsection{Quasi $p$-Steinberg characters of Symmetric groups}
It is well-known that the set of integer partitions of $n$ indexes conjugacy classes as well as irreducible characters of $S_n$. Given a partition $\mu = (\mu_1, \mu_2, \ldots, \mu_{l(\mu)})$ of $n$ (denoted by $\mu \vdash n$), let $\varkappa^{\mu}$ denote the corresponding irreducible character of $S_n$.  All triples $(n,\mu,p)$ such that $\varkappa^{\mu} $ is a quasi $p$-Steinberg character of $S_n$ are listed in Table~\ref{tab:quasi-Sn}. See \cite[Theorem 1.3]{PS22} for the proof.

\begin{table}[ht] 
	\begin{tabular}[t]{c|c|c}
		\,\,\,\,\, \,\,\, $n$ \,\,\,\,\, \,\,\,& \,\,\,\,\, \,\,$\mu$  \,\,\,\,\, \,\, &\,\,\,\,\, \,\,\, $p$  \,\,\,\,\, \,\,\, \\
		\hline \hline 
		$3$ & $(2,1)$ & $2$ \\
		\hline 
		$4$ & $(2,2)$ & $2$ \\
		$4$ & $(3,1)$, $(2,1,1)$  & $3$ \\
		\hline 
		$5$ &  $(4,1)$, $(2,1,1,1)$ & $2$ \\
		$5$ & $(3,2)$, $(2,2,1)$ & $5$ \\
		\hline 
		$6$ & $(3,2,1)$ & $2$ \\
		$6$ & $(4,2)$, $(2,2,1,1)$ & $3$ \\
		\hline 
		$8$ & $(5,2,1), (3,2,1,1,1)$ & $2$ \\
		\hline 
	\end{tabular}
	\vspace{.2cm} 
	\caption{quasi $p$-Steinberg characters of $S_n$}\label{tab:quasi-Sn}
\end{table}

	\subsection{Complex reflection groups $\Grqn$}
	\label{pre}
	
	Given positive integers $r$ and $n$, the symmetric group $S_n$ acts on the direct product $\Zr^n$ of $n$ copies of the additive cyclic group $\Zr$ by permuting the coordinates. This action gives us the wreath product of $\Zr$ by $S_n$,  denoted by $G(r,1,n) := \Zr^n \rtimes S_n$, i.e., 
	\begin{equation*}
		G(r,1,n)  = \{(z_1, z_2, \ldots, z_n; \sigma) \mid z_i \in \Zr \mbox{ for all } 1 \leq i \leq n, \sigma \in S_n\}.
	\end{equation*}  
	
	For a positive integer $q$ which divides $r$, we define a subgroup $G(r,q,n)$ of $G(r,1,n)$ as follows: $$ G(r,q,n) := \{(z_1, z_2, \ldots, z_n; \sigma) \in G(r,1,n) \mid \sum\limits_{i=1}^{n}z_i \equiv 0 \mkern-18mu\pmod{q}\}.$$
	The group $G(r,q,n)$ is a normal subgroup of $G(r,1,n)$ of index $q$. By Shephard--Todd classification, the family of groups $\Grqn$ for $n > 1$, (except the group $G(2,2,2)$), is the only infinite family of finite irreducible imprimitive \emph{complex reflection groups} \cite[Section 2]{ST54}. The group $G(2,2,2)$ is imprimitive, but it is not irreducible \cite[Theorem (2.4)]{C76}. The group $G(r,1,n)$ is also known as the \emph{generalized symmetric group}.
	
	Some families of groups which are special cases of $\Grqn$ are:
	\begin{enumerate}[(a)]
		\item Cyclic group of order $r$, $ \Z/r\Z = G(r,1,1)$;
		\item Dihedral group of order $2r$, $D_{2r} = G(r,r,2)$;
		\item Symmetric group on $n$ symbols, $S_n = G(1,1,n)$;
		\item Weyl group of type $B_n$ (also called hyperoctahedral group) is $G(2,1,n)$; and
		\item Weyl group of type $D_n$ is $G(2,2,n)$. 
	\end{enumerate}

	\subsection{Conjugacy classes of $G(r,q, n)$}\label{conj}
	Let $\pi = (z_1, z_2, \ldots, z_n; \sigma) \in G(r,1,n)$ be such that the corresponding cycle decomposition of $\sigma$  be $c = (c_1, c_2, \ldots, c_t)$, where the cycles are written in an arbitrarily fixed order. For $1 \leq i \leq t$, suppose that the cycle $c_i$ of length $l(c_i)$ is written as $(c_{i1}, c_{i2}, \ldots, c_{il(c_i)})$. We define the \emph{color of the cycle} $c_i$ to be the cycle sum $z(c_i) := z_{c_{i1}}+z_{c_{i2}}+\cdots+z_{c_{il(c_i)}} \in \mathbb{Z}_r$. Let $\mathcal P$ denote the set of all partitions (by convention, $0$ has a unique partition, called empty partition, denoted by $\emptyset$). For $\pi = (z_1, z_2, \ldots, z_n; \sigma) \in G(r,1,n)$, define a map $$\tau_\pi : \Zr \rightarrow \mathcal P$$ by setting $\tau_\pi(j) $ to be the partition associated to the multiset of lengths of all cycles in $\sigma$ whose color is $j$; denote this partition by $\bla_j$. The map $\tau_\pi$ is called the \emph{type} of $\pi$. Thus, the type of $\pi$ can be written as an $r$-tuple of partitions, $\bla = (\bla_0, \bla_1, \ldots, \bla_{r-1})$, such that the total sum of all the parts is $n$. We call such $r$-tuple of partitions an \emph{$r$-partite partition} of size $n$. Given a partition, we consider Young diagram associated to it. Throughout this article, we use the notions of partition and Young diagram interchangeably. Let $\Y(r,n)$ denote the \emph{set of all $r$-partite Young diagrams} with total number of boxes being $n$. The following theorem (see \cite[p.170]{Mac95} for a proof) states that the conjugacy classes of $G(r,1,n)$ are indexed by the set $\Y(r,n)$.
	
	\begin{theorem}
		Two elements $\pi_1$ and $\pi_2$ in $G(r,1,n)$ are conjugate if and only if the corresponding types are equal, i.e., $\tau_{\pi_1} = \tau_{\pi_2}$. 
	\end{theorem}
	
	Let $z(\pi)$ be the gcd of all nonzero $z(c_i)$ for $1 \leq i \leq t$. Define $$d(\pi) := \mbox{gcd} (z(\pi), l(c_1), l(c_2), \ldots, l(c_t), q).$$ The following theorem describes the splitting of conjugacy classes of $G(r,1,n)$ into conjugacy classes of $\Grqn$; for more details about conjugacy classes of $\Grqn$, see \cite[Section 2]{Read77}. 
	
	\begin{theorem}{ \cite[Theorem 3]{Read77}}\label{split}
		If $\pi  \in G(r,q,n)$, then the conjugacy class of $\pi$ in $G(r,1,n)$ splits into $d(\pi)$ conjugacy classes in $\Grqn$. 
	\end{theorem}
	
	Now, we give a characterization of an element in $G(r,1,n)$ to be an element of $\Grqn$. 
	\begin{lemma}
		Given $\pi = (z_1, z_2, \ldots, z_n; \sigma) \in G(r,1,n)$, let $\bla = (\bla_0, \bla_1, \ldots, \bla_{r-1})$ be the type of $\pi$. Then, $\pi$ belongs to $\Grqn$ if and only if $\sum\limits_{j=1}^{r-1} jl(\bla_j) \equiv 0 (\mbox{mod}~q)$.
	\end{lemma}
	
	\begin{proof}
		The parts of the partition $\bla_j$ are the lengths of the cycles in $\sigma$ whose color is $j$. For each cycle $(i_1,  \ldots, i_k)$ of color $j$, we have $z_{i_1}+ \cdots + z_{i_k} = j$. Thus, $j$ appears $l(\bla_j)$ times in the sum $z_1+\cdots+z_n$. Therefore,
		$$z_1+\cdots+z_n = \sum\limits_{j=0}^{r-1} jl(\bla_j) = \sum\limits_{j=1}^{r-1} jl(\bla_j).$$
	\end{proof}
	
	\subsection{Representation theory of $\Grqn$}
	The representation theory of wreath product of a finite group by symmetric group $S_n$ is a well-studied subject, see \cite{cst14, jk81, Mac95, MS16}. As stated in Section \ref{conj}, the set $\Y(r,n)$ indexes the irreducible representations of $G(r,1,n)$. The dimension of an irreducible representation $V^{\bla}$ corresponding to $\bla \in \Y(r,n)$ can be obtained from \cite[Theorem 6.7]{MS16}.
	
	Since $\Grqn$ is a normal subgroup of $G(r,1,n)$, the representation theory of $\Grqn$ can be deduced from the representation theory of $G(r,1,n)$ using Clifford theory. Let $H$ denote the group of one-dimensional representations of $G(r,1,n)$ which contain $\Grqn$ in their kernel. The group $H$ acts on the set of irreducible representations of $G(r,1,n)$. Given $\bla \in \Y(r,n)$, let $[\bla]$ and $H_{\bla}$ denote the orbit of $\bla$ and the stabilizer subgroup of $H$ with respect to $\bla$, respectively.  
	
	Given $m := \frac{r}{q}$, let us define a combinatorial object, an \emph{$(m,q)$-necklace} \cite[p.174]{HR98}, which will be useful in parametrization of irreducible $\Grqn$-modules. Given $\bla = (\bla_0, \bla_1, \ldots, \bla_{r-1}) \in \Y(r,n)$, consider the $q$-tuple 
	$$\tla_{(i)} := (\bla_{i}, \bla_{m+i}, \bla_{2m+i}, \ldots, \bla_{(q-1)m+i}),$$ 
	for each $0 \leq i \leq m-1$. We depict $\tla_{(i)}$ as a circular $q$-necklace in $(x,y)$-plane with the $\bla_{i}$ node being placed on the positive $y$-axis and the $j$-th node $\bla_{(j-1)m+i}$, $2 \leq j \leq q$, being placed at a clockwise angle of $2\pi/(j-1)$ with  the positive $y$-axis. An {\it $(m,q)$-necklace} of total $n$ boxes obtained from $\bla \in \Y(r,n)$, denoted by $\tbla$, is an $m$-tuple of $q$-necklaces $$ \tbla = (\tla_{(0)}, \tla_{(1)}, \ldots, \tla_{(m-1)}).$$ 
	For $ 1 \leq j \leq q$ and $0 \leq i \leq m-1$, let $\tla_{(i,j)}$ denote the $j$-th node in $\tla_{(i)}$, i.e., $\tla_{(i,j)} = \La_{(j-1)m+i}$.  Thus, 
	\begin{equation*}
		\sum\limits_{i=0}^{m-1}\sum\limits_{j=1}^q\tla_{(i,j)} = n.
	\end{equation*}

	\begin{example}\label{exn}
		An example of $(3,4)$-necklace of total $~24$ boxes obtained from $$\bla = ((2,1), (2,2), (1,1), (1), (1,1), \emptyset, (2,1), (2,2), (1,1), (1), (1,1), \emptyset ) \mbox{ is}:$$
		
		
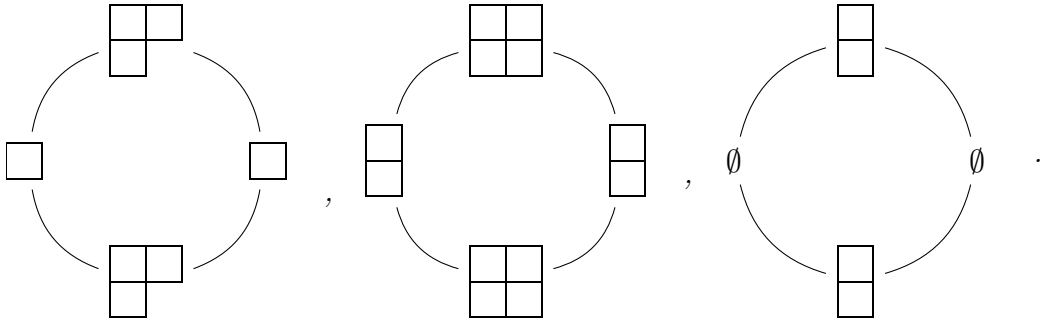
\begin{center}
		 \begin{tikzpicture}[scale=0.8,mycirc/.style={circle,fill=black, minimum size=0.1mm, inner sep = 1.5pt}]
			
			\node (1) at (0,2) {$\yng(2,1)$};
			\node (2) at (2,0) {$\yng(1)$};
			\node (3) at (0,-2) {$\yng(2,1)$};
			\node (4) at (-2,0) {$\yng(1)$};
			\draw  (1) to[bend left=30] (2);
			\draw  (2) to[bend left=30] (3);
			\draw  (3) to[bend left=30] (4);
			\draw   (4) to[bend left=30] (1);
			\node(a) at (3,-0.7) {,};
		\end{tikzpicture}
		\begin{tikzpicture}[scale=0.8,mycirc/.style={circle,fill=black, minimum size=0.1mm, inner sep = 1.5pt}]
			\node (1) at (0,2) {$\yng(2,2)$};
			\node (2) at (2,0) {$\yng(1,1)$};
			\node (3) at (0,-2) {$\yng(2,2)$};
			\node (4) at (-2,0) {$\yng(1,1)$};
			\draw  (1) to[bend left=30] (2);
			\draw  (2) to[bend left=30] (3);
			\draw  (3) to[bend left=30] (4);
			\draw   (4) to[bend left=30] (1);
			\node(a) at (3,-0.4) {,};
		\end{tikzpicture}
		\begin{tikzpicture}[scale=0.8,mycirc/.style={circle,fill=black, minimum size=0.1mm, inner sep = 1.5pt}]
			\node (1) at (0,2) {$\yng(1,1)$};
			\node (2) at (2,0) {$\emptyset$};
			\node (3) at (0,-2) {$\yng(1,1)$};
			\node (4) at (-2,0) {$\emptyset$};
			\draw  (1) to[bend left=30] (2);
			\draw  (2) to[bend left=30] (3);
			\draw  (3) to[bend left=30] (4);
			\draw   (4) to[bend left=30] (1);
			\node(a) at (3,0) {.};
		\end{tikzpicture}
	\end{center}
	\end{example}
	
	Two $(m,q)$-necklaces, $\tbla$ and $\tilde{\bm{\mu}}$, both with total $n$ boxes, are said to be {\em equivalent} if for some integer $t$, $\tla_{(i,j)} = \tilde{\bm\mu}_{(i,(j+t)(mod~q))}$ for all $1 \leq j \leq q$ and $0 \leq i \leq m-1$. Let $\Y(m,q,n)$ denote the \emph{set of inequivalent $(m,q)$-necklaces} with total $n$ boxes.
	
	Theorem \ref{irrgp} describes the inequivalent irreducible representations of $\Grqn$. For more details and proof of Theorem \ref{irrgp}, see \cite{S89, MY98, BB07, MS20}.
	
	\begin{theorem}\label{irrgp}
		The irreducible $\Grqn$-modules are indexed by the set of ordered pairs $(\tbla, \delta)$, where $\tbla \in \Y(m,q,n)$ and $\delta \in H_{\bla}$. Given $\bla \in \Y(r,n)$, the restriction of the corresponding irreducible $G(r,1,n)$-module $V^{\bla}$  to $\Grqn$ has multiplicity free decomposition given as:
		$$ \mbox{\em Res}_{G(r,q,n)}^{G(r,1,n)}(V^{\bla}) \; = \; \bigoplus_{\delta \in H_{\bla}} V^{(\tbla,\delta)},$$ where  $V^{(\tbla,\delta)}$ denotes the irreducible $G(r,q,n)$-module indexed by $(\tbla, \delta)$. Moreover, for $\bla, \bm\mu \in \Y(r,n)$, we have   $\mbox{\em Res}_{G(r,q,n)}^{G(r,1,n)}(V^{\bla}) \; \cong \; \mbox{\em Res}_{G(r,q,n)}^{G(r,1,n)}(V^{\bm \mu})$ if and only if ${\bm \mu} \in [\bla]$.
	\end{theorem}
	
	Next lemma is an important consequence of Clifford theory and will be helpful in studying character values of an element in $G(r,1,n)$ whose conjugacy class does not split in the normal subgroup $\Grqn$.
	
	\begin{lemma}\label{nosplit}
		Let $N$ be a normal subgroup of a group $G$. Suppose that $\rho$ and $\tau$ are irreducible representations of $G$ and $N$ with characters $\chi_\rho$ and $\chi_\tau$, respectively, such that $\langle \chi_\rho|_N, \chi_\tau\rangle_N \neq 0$. Assume that $x$ is an element of $N$ such that its conjugacy class in $G$ does not split in $N$. Then, $\chi_\rho|_N(x)$ is a positive integral multiple of $\chi_\tau(x)$.
	\end{lemma}
	
	\begin{proof}
		Assume that $\{g_1=id, g_2, \ldots, g_k\}$ be a set of coset representatives of the inertia group of $\tau$ in $G$. 
		By Clifford theory, $\rho|_N = \displaystyle\oplus_{i=1}^{k} (\tau^{g_i})^{f} = (\displaystyle\oplus_{i=1}^{k} \tau^{g_i})^{f}$, where $\tau^{g_i}$ is a conjugate representation of $\tau$ and $f$ is the multiplicity of each conjugate representation. Then the result follows from the following computation:
		\begin{align*}
			\chi_\rho|_N(x) = f\sum\limits_{i=1}^{k}\chi_\tau^{g_i}(x) 
			= f \sum\limits_{i=1}^{k}\chi_\tau(g_i^{-1}xg_i)
			= f \sum\limits_{i=1}^{k}\chi_\tau(x) = fk\chi_\tau(x).
		\end{align*}
	\end{proof}
	
	\subsection{Murnaghan--Nakayama rule for $G(r, 1,n)$}

	The Murnaghan--Nakayama rule \cite[Theorem 4.3]{S89} is a combinatorial method to compute the irreducible characters of $G(r,1,n)$. An edge connected skew shape which does not contain a $2 \times 2$ square of boxes is called a \emph{ribbon}. Given a partition $\mu$ of $n$, a \emph{ribbon tableau} of shape $\mu$ is obtained by filing the boxes of $\mu$ with positive integers such that the entries in each row and column are weakly increasing and each appearing integer forms a distinct ribbon.  Let  $\bla  = (\bla_0, \bla_1, \ldots, \bla_{r-1})$ be an $r$-partite partition of $n$. An $r$-tuple $T=(T_0,\ldots, T_{r-1})$ is called \emph{$r$-partite ribbon tableau} of shape $\bla$ if each $T_j$ is a ribbon tableau of shape $\bla_j$ and for each positive integer $i$, the ribbon containing $i$ appears in at most one  component of $T$. 
	
	Given an $r$-partite ribbon tableau $T=(T_0,\ldots, T_{r-1})$ we denote its $i$-th \emph{index}, $i$-th \emph{length} and $i$-th \emph{height} by $f_T(i), l_T(i)$ and $ht_T(i)$, respectively, which are defined as follows:
	
	\begin{tabular}{ll}
		$f_T(i) :=$ index of the ribbon containing $i$ in $T$; \\	
		$l_T(i) :=$ size of the ribbon containing $i$ in $T$; \\
		$ht_T(i) :=$ one less than the number of rows in the ribbon containing $i$ in $T$.
    \end{tabular}

	\begin{theorem}{\cite[Proposition 2.2]{APR10}}(Murnaghan--Nakayama rule for $G(r,1,n)$)
		\label{MNR}
		Suppose that $\chi^{\bla}$ denotes the irreducible character of $G(r,1,n)$ corresponding to $\bla \in \Y(r,n)$. Let $\pi = (z_1, z_2, \ldots, z_n; \sigma) \in G(r,1,n)$ such that the corresponding cycle decomposition of $\sigma$  in an arbitrarily fixed order be $c = (c_1, c_2, \ldots, c_t)$. For $1 \leq i \leq t$, let $l(c_i)$ and $z(c_i)$ be the length and color of the cycle $c_i$, respectively. Then,
		\begin{equation}\label{eqmnr}
			\chi^{\bla}(\pi) = \sum\limits_{T \in RT_c(\bla)}\prod\limits_{i=1}^{t}(-1)^{ht_T(i)}\omega^{f_T(i).z(c_i)},
		\end{equation}
		where $RT_c(\bla)$ is the set of all $r$-partite ribbon tableaux $T$ of shape $\bla$ such that $l_T(i) = l(c_i)$ for all $ 1 \leq i \leq t$, and $\omega = e^{2\pi \iota/r}$. 
	\end{theorem}
 
Let us illustrate Theorem \ref{MNR} by an example.
\begin{example}
	Consider the $3$-partite partition $\bla = ((2,1), \emptyset, (1,1,1))\in \Y(3,6)$ and the element $\pi$ of $G(3,1,6)$ 
	$$\pi = (1,1,0,0,1,0;(1,2,3)(4,5)).$$ Here $c=(c_1,c_2,c_3)$, where we choose $c_1=(1,2,3), c_2 = (4,5)$ and $c_3=(6)$. The set $RT_c(\bla)$ consists of three $3$-partite ribbon tableaux $T=(T_0,T_1,T_2)$ as described below
	\begin{displaymath}
	\ytableausetup{centertableaux}
	\ytableaushort
	{11,1}, 
	\hspace{0.1cm}
	\emptyset,
	\hspace{0.1cm}
	\ytableausetup{centertableaux}
	\ytableaushort
	{2,2,3} \hspace{0.1cm};
	\hspace{0.5cm}
	\ytableaushort
	{22,3}, 
	\hspace{0.1cm}
	\emptyset,
	\hspace{0.1cm}
	\ytableausetup{centertableaux}
	\ytableaushort
	{1,1,1} \hspace{0.1cm};
	\hspace{0.5cm}
	\ytableaushort
	{23,2}, 
	\hspace{0.1cm}
	\emptyset,
	\hspace{0.1cm}
	\ytableausetup{centertableaux}
	\ytableaushort
	{1,1,1} \hspace{0.1cm}.
	\end{displaymath}
	Thus, the character value is $$\chi^{\bla}(\pi) = (-1)(-1)\omega^{2.1}+(-1)^2\omega^{2.2}+(-1)^2\omega^{2.2}(-1) = \omega^2+\omega-\omega = \omega^2.$$
	Now consider another element of $G(3,1,6)$:
	$$\pi=(1,1,0,0,1,0;(1,2)(3,5)(4,6)).$$ 
	Here $l(c_i)=2$ for all $i=1,2,3.$
Observe that there are no $3$-partite ribbon tableaux $T=(T_0,T_1,T_2)$ of shape $\bla$ such that $l_T(i)=2$. Hence the character value $\chi^{\bla}(\pi)$ is zero.
\end{example}
	
	In the following corollary of Theorem \ref{MNR}, we relate character values of $G(r,1,n)$ with those of $S_n$.
	\begin{corollary}\phantomsection\label{MN cor}
		\begin{enumerate}[(i)]
			\item For an $r$-partite partition $\bla = (\bla_0, \bla_1, \ldots, \bla_{r-1})$ of $n$, the character value $\chi^{\bla}((z_1, z_2, \ldots, z_n; (1,2,\ldots,n)))$ is nonzero if and only if $\bla=\widehat{\bla}^j$ for some $0 \leq j \leq r-1$ with $\bla_j$ being a hook of size $n$.
			 In such a case, we have 
			$$\chi^{\bla}((z_1, z_2, \ldots, z_n; (1,2,\ldots,n)))= \omega^{j(z_1+\cdots+z_n)}\varkappa^{\bla_j}((1,2,\ldots, n)).$$
			 
			\item For an element $\sigma$ of $S_n$ and $0 \leq j \leq r-1$, we have $$ \chi^{\widehat{\bla}^j}((0,\ldots,0; \sigma)) = \varkappa^{\bla_j}(\sigma).$$ 
			\item For $\pi=(z_1,\ldots,z_n;\sigma)\in G(r,1,n)$, we have $$\chi^{\widehat{\bla}^0}(\pi) = \varkappa^{\bla_0}(\sigma).$$
			Moreover, for $0 \leq j,k \leq r-1$, $\chi^{\widehat{\bla}^j}(\pi)$ is nonzero if only if $\chi^{\widehat{\bla}^k}(\pi)$ is nonzero.
			
			\item The character $ \chi^{\widehat{\bla}^j} $ never vanishes on $(z_1, z_2, \ldots, z_n; (1))  \in G(r,1,n)$. 
		\end{enumerate}
	\end{corollary}

\begin{proof}\begin{enumerate}[(i)]
		\item Let $\chi^{\bla}((z_1, z_2, \ldots, z_n; (1,2,\ldots,n)))$ is nonzero. Since $ \sigma = (1,2,\ldots,n)$ has a single part of length $n$, $T$ must have only one non empty component, say $T_j$, which must be a single ribbon of size $n$. That implies $\bla=\widehat{\bla}^j$ and $\bla_j$ is a hook of size $n$. The converse and the computation of character value are straightforward.

\item Here, $\widehat{\bla}^j$ has only one nonempty component and $z_i = 0$ for $i=0,1, \ldots, r-1$. We obtain the result by evaluating equation \eqref{eqmnr} and comparing it with Murnaghan--Nakayama rule for $S_n$ (\cite[Theorem 7.17.3]{MR1676282}).
		
		\item Since $T_j=\emptyset$ for all $j\neq 0$, $f_T(i)=0$ for all $i$. So \eqref{eqmnr} is equivalent to Murnaghan--Nakayama rule to compute irreducible characters of $S_n$ corresponding to the partition $\bla_0$. For the next part, a routine calculation shows that $$\chi^{\widehat{\bla}^j}(\pi)= \omega^{j-k}\chi^{\widehat{\bla}^k}(\pi).$$
\item When $\sigma = (1)$, the set $RT_c(\bla)$ is nonempty and for each $r$-partite ribbon tableau $T \in RT_c(\bla)$, we have $ht_T(i) = 0, f_T(i)=j, z(c_i) = z_i$ for all $i=1,2,\ldots, n$. 	
\end{enumerate}
\end{proof}

Next, we write some observations as a lemma which will be useful later.

	\begin{lemma}\phantomsection\label{proj}
	\begin{enumerate}[(i)]
		\item If $\pi = (z_1, z_2, \ldots, z_n; \sigma)$ is a $p$-regular element in $G(r,1,n)$, then $\sigma$ is a $p$-regular element of $S_n$. 
		\item The order of sum of elements of odd order in an abelian group can never be even.
	\end{enumerate}
\end{lemma}

\begin{proof}
	\begin{enumerate}[(i)]
		\item This follows from the observation that under the projection map $$pr: G(r,1,n) \rightarrow S_n,$$ the order of $ pr(z_1, z_2, \ldots, z_n; \sigma)$ divides the order of $(z_1, z_2, \ldots, z_n; \sigma)$. 
		\item Let $g_1, \ldots, g_s$ be elements of odd orders $l_1, \ldots, l_s$ respectively, in an abelian group. Then the order of $g_1+\cdots+g_s$ divides $l_1 \cdots l_s$. 
	\end{enumerate}
\end{proof}

	\section{Proof of Theorems~\ref{p23}, ~\ref{MTq1}, and ~\ref{MTqnot1-hefty}}
	\label{sec:main-results}

	We now prove the results stated in Section \ref{intro} for the groups $G(r,1,n)$ and $G(r,q,n)$. The groups $G(r,1,1)$ and $G(r,q,1)$ are cyclic all of whose irreducible characters are one-dimensional and thus, quasi $p$-Steinberg. Therefore, now onwards we assume that $n \geq 2$ for $G(r,1,n)$ and $G(r,q,n)$.
	
	\begin{proof}[{\bf Proof of Theorem~\ref{p23}}]
		\begin{enumerate}[(i)]
\item For $n \in \{2,3\}$, consider an $n$-regular element $g$, say of order $t$, in $G$ and an $n$-dimensional irreducible character $\chi$ of $G$. Then, $\chi(g)$ is a sum of $t$-th roots of unity and there are $n$-summands in this sum. Such a sum can never be zero \cite[Table 3.1, p.141]{PR98}. For $p=2$ and $n=4$, the character value of a $4$-dimensional irreducible character on a $2$-regular element $g$ of odd order $t$ is a sum of four summands of $t$-th roots of unity. Such a sum can not be zero because a sum of four $t$-th roots of unity can be zero only when at least one of the roots of unity is of even order.
\item This follows from the fact that $g \in G$ is $p$-regular if and only if $\mathfrak{a}(g)$ is $p$-regular.
		\end{enumerate}	
	\end{proof}
	
	\begin{remark}
	\label{rk:p23}  It is interesting to note that Theorem~\ref{p23}(i) is not true in general for any $n \geq 5$. The $n$-dimensional irreducible character $\chi^{(n,1)}$ of $S_{n+1}$ is not a quasi $p$-Steinberg character for any prime $p \leq n+1 $ (see Table \ref{tab:quasi-Sn}).
	\end{remark}

	As an immediate application of Theorem~\ref{p23}(i), we list the triples $(n,\bla_j,p)$ in Table~\ref{tab:specific-quasi} which correspond to some specific quasi $p$-Steinberg characters $\chi^{\widehat{\bla}^j}$ of $G(r,1,n)$.
	 
\begin{table}[ht] 
\begin{tabular}[t]{c|c|c|c}
	 \,\,\,\,\, \, $n$ \,\,\,\,\, \ &\,\,\,\,\, \,\,\, $\bla_j$  \,\,\,\,\, &\,\,\,$\dim(\chi^{\widehat{\bla}^j})$\,\,\,\,\,\, & \,\,\,\,\,$p$  \,\,\,\,\, \\
	 \hline \hline
	 $3$  & $(2,1)$ & $2$ & $2$ \\
	 \hline
	 $4$  & $(3,1)$   & $3$ & $3$ \\
	 $4$  & $(2,1,1)$ & $3$ & $3$ \\
	 $4$  & $(2,2)$ & $2$ & $2$ \\
	 \hline
	 $5$  & $(4,1)$ & $4$& $2$ \\
	 $5$  & $(2,1,1,1)$ & $4$ & $2$ \\
	 \hline
	 \end{tabular}
	 \vspace{.2cm} 
			\caption{Some specific quasi $p$-Steinberg characters of $G(r,1,n)$}\label{tab:specific-quasi}
	\end{table}				
\begin{prop}\label{reduction}
Given an $r$-partite partition $\bla = (\bla_0, \bla_1, \ldots, \bla_{r-1})$ of $n$ and a prime $p$, assume that $\chi^{\bla}$ is a quasi $p$-Steinberg character of $G(r,1,n)$. Then, either $\bla=\widehat{\bla}^j$ for some $0 \leq j \leq r-1$ or $\bla=\widehat{\bla}^{j,k}$ for some $0 \leq j\neq k\leq r-1$.
\end{prop}
\begin{proof}
When $p\nmid n$, the element $\alpha=(0, 0, \ldots, 0; (1,2,\ldots,n))$ is $p$-regular in $G(r,1,n)$. Since $\chi^{\bla}$ is a quasi $p$-Steinberg character, we have $\chi^{\bla}(\pi)\neq 0$. By Corollary \ref{MN cor}(i), we must have $\bla=\widehat{\bla}^j$ for some $0 \leq j \leq r-1$.

When $p\mid n$, consider the $p$-regular element $\alpha_1=(0, 0, \ldots, 0; (1,2,\ldots,n-1))$. By Murnaghan--Nakayama rule for $G(r,1,n)$, the only  possible $\bla = (\bla_0, \bla_1, \ldots, \bla_{r-1})$ such that $\chi^{\bla}(\pi)\neq 0$ must satisfy $\bla = \widehat{\bla}^j$ for some $0 \leq j \leq r-1$ or $\bla=\widehat{\bla}^{j,k}$ for some $0 \leq j\neq k\leq r-1$.
\end{proof}

	\begin{prop}\label{QGS}
		Let $ \mu$ be a partition of $n$ such that  $ \mu \notin \{(n), (1^n)\}$.  The character $ \chi^{\widehat{\bla}^j}$, where $\bla_j=\mu$, is a quasi $p$-Steinberg character of $G(r,1,n)$ if and only if $\varkappa^\mu$ is a quasi $p$-Steinberg character of $S_n$. 
	\end{prop}
	
	\begin{proof}
		Let $\chi^{\widehat{\bla}^j}$, where $\bla_j=\mu$, be a quasi $p$-Steinberg character of $G(r,1,n)$. Consider a $p$-regular element $\sigma$ of $S_n$. Then, Corollary \ref{MN cor}(ii) implies that 
		$$\varkappa^{\mu}(\sigma) =\chi^{\widehat{\bla}^j}((0,\ldots,0; \sigma)),$$
		which is nonzero because $(0,\ldots,0; \sigma)$ is a $p$-regular element of $G(r,1,n)$.
		
		Conversely, assume that $\varkappa^\mu$ is a quasi $p$-Steinberg character of $S_n$. Let  $\pi=(z_1,\ldots,z_n;\sigma)$ be a $p$-regular element of $G(r,1,n)$. 
Then, $\sigma$ is a $p$-regular element of $S_n$ by Lemma \ref{proj}(i), and we have $\varkappa^\mu(\sigma)\neq 0$. Now by Corollary \ref{MN cor}(ii), we have

\begin{displaymath}
\chi^{\widehat{\bla}^0}(\pi)=\varkappa^\mu(\sigma).
\end{displaymath}
 and
 $\chi^{\widehat{\bla}^j}(\pi)$ is nonzero if and only if $\chi^{\widehat{\bla}^0}(\pi)$ is nonzero for any $j$ 
Therefore, $\chi^{\widehat{\bla}^j}$ is a quasi $p$-Steinberg character of $G(r,1,n)$. 
\end{proof}		

\begin{proof}[{\bf Proof of Theorem~\ref{MTq1}}]		
		Let $\chi^{\bla}$ be a quasi $p$-Steinberg character of $G(r,1,n)$, where $\bla = (\bla_0, \bla_1, \ldots, \bla_{r-1})$. By Proposition~\ref{reduction}, we have either $\bla=\widehat{\bla}^j$ for some $0 \leq j \leq r-1$ or $\bla=\widehat{\bla}^{j,k}$ for some $0 \leq j\neq k\leq r-1$.
		
		Using Proposition~\ref{QGS} and Table~\ref{tab:quasi-Sn}, we get the quasi $p$-Steinberg characters of $G(r,1,n)$ corresponding to $\widehat{\bla}^j$ in Table \ref{tab:quasi-Gr1n}. 
		
		Now we shall classify quasi $p$-Steinberg characters of $G(r,1,n)$ corresponding to $\widehat{\bla}^{j,k}$. As we observe in the proof of Proposition~\ref{reduction}, this kind of character may appear only when $p \mid n$.
	
When $n \geq 5$, there does not arise a quasi $p$-Steinberg character of this kind because of the following observations:
			\begin{enumerate}[(i)]
				\item Either the element $\alpha_2 = (0, \ldots, 0;(1,2,\ldots, n-2)(n-1,n))$ or the element $\alpha_3 = (0, \ldots, 0;(1,2,\ldots, n-3)(n-2,n-1,n))$ is a $p$-regular element of $G(r,1,n)$;
				
				\item $\chi^{\bla}(\alpha_2) = \chi^{\bla}(\alpha_3) = 0$. 
			\end{enumerate}

Now we consider $n < 5$ in which some additional quasi $p$-Steinberg characters for $G(r,1,n)$ will arise.

Let us start with $n=2, p = 2$. The irreducible character corresponding to $\bla=\widehat{\bla}^{j,k}$, where $\bla^j = (1),$ is a quasi $2$-Steinberg character by Theorem \ref{p23}(i). 
	
Let us consider the case $n=3, p = 3$. The irreducible characters corresponding to $\widehat{\bla}^{j,k}$, where $\lambda_j \vdash 2$, are three-dimensional and hence, quasi $3$-Steinberg characters of $G(r,1,3)$ by Theorem \ref{p23}(i).

Now consider the case $n=4, p=2$. The irreducible characters corresponding to $\bla=\widehat{\bla}^{j,k}$ where $\bla_j \in \{(3), (1,1,1)\}$, are quasi $2$-Steinberg characters of $G(r,1,4)$ by Theorem \ref{p23}(i) by virtue of being four-dimensional characters.
		
We now discuss the final case: $n=4, p=2$ and  $\bla=\widehat{\bla}^{j,k}$ with $\bla_j = (2,1)$. For an $r$-partite ribbon tableau $T=(T_0,\ldots, T_{r-1})$ of shape $\widehat{\bla}^{j,k}$, we have $T_l=\emptyset$ for all $ l \notin \{j,k\}$. Let $(z_1, \ldots, z_4; \sigma)$ be a $p$-regular element in $G(r,1,4)$, thus, either $\sigma = (1)$ or $\sigma$ has cycle type $(3,1)$, and $z_1, \ldots, z_4$ are of odd orders. When $\sigma = (1)$, then there are exactly eight $r$-partite ribbon tableaux $T=(T_0,\ldots, T_{r-1})$ as we have eight different pairs $(T_j,T_k)$:

\begin{displaymath}
\ytableausetup{centertableaux}
		\ytableaushort
		{12,3}
		\hspace{0.3cm}
		\ytableausetup{centertableaux}
		\ytableaushort
		{4}~;
		\hspace{1cm}
		\ytableaushort
		{13,2}
		\hspace{0.3cm}
		\ytableausetup{centertableaux}
		\ytableaushort
		{4}~;
		\hspace{1cm}
		\ytableausetup{centertableaux}
		\ytableaushort
		{12,4}
		\hspace{0.3cm}
		\ytableausetup{centertableaux}
		\ytableaushort
		{3}~;
		\hspace{1cm}
		\ytableausetup{centertableaux}
		\ytableaushort
		{14,2}
		\hspace{0.3cm}
		\ytableausetup{centertableaux}
		\ytableaushort
		{3}~;
		\end{displaymath}
\begin{displaymath}
\ytableausetup{centertableaux}
		\ytableaushort
		{13,4}
		\hspace{0.3cm}
		\ytableausetup{centertableaux}
		\ytableaushort
		{2}~;
		\hspace{1cm}
\ytableausetup{centertableaux}
		\ytableaushort
		{14,3}
		\hspace{0.3cm}
		\ytableausetup{centertableaux}
		\ytableaushort
		{2}~;
		\hspace{1cm}
		\ytableausetup{centertableaux}
		\ytableaushort
		{23,4}
		\hspace{0.3cm}
		\ytableausetup{centertableaux}
		\ytableaushort
		{1}~;
		\hspace{1cm}
		\ytableausetup{centertableaux}
		\ytableaushort
		{24,3}
		\hspace{0.3cm}
		\ytableausetup{centertableaux}
		\ytableaushort
		{1}~.
\end{displaymath}

By Murnaghan--Nakayama rule for $G(r,1,n)$, we have 
\begin{align*}
  & \chi^{\bla}((z_1,z_2, z_3, z_4, (1))) \\ 
  & = 2(\omega^{jz_1+jz_2+jz_3+kz_4}+ \omega^{jz_1+jz_3+jz_4+kz_2} + \omega^{jz_1+jz_2+jz_4+kz_3}+\omega^{kz_1+jz_2+jz_3+jz_4}) \\ 
  &=  2\omega^{kz_1+jz_2+jz_3+jz_4}(\omega^{(j-k)(z_1-z_4)}+\omega^{(j-k)(z_1-z_2)}+\omega^{(j-k)(z_1-z_3)}+1),
\end{align*}
which can be zero only when $r$ is an even integer \cite[Main Theorem, p.2]{LL00}. However, when $r$ is even, the sum above is nonzero by Lemma \ref{proj}(ii) because $z_1, z_2, z_3, z_4$ have odd orders in $\Z/r\Z$ \cite[Table 3.1, p.141]{PR98}.

When $\sigma$ is of cycle type $(3,1)$, say $\sigma = (1,2,3)$, then there is exactly one $r$-partite ribbon tableaux $T$ given by 
\begin{displaymath}
T_j = \begin{ytableau}
	1&1\\1
\end{ytableau}
\hspace{0.5cm}
T_k = \begin{ytableau}
	2
\end{ytableau}
\hspace{0.5cm}
T_l = \emptyset \text{ for all } l \neq j, k.
\end{displaymath}

This implies that $\chi^{\bla}((z_1,z_2, z_3, z_4, (1,2,3)))  = -\omega^{jz_1+jz_2+jz_3+kz_4} \neq 0.$
\end{proof}

We now describe the quasi $p$-Steinberg characters of $G(r,q,n)$ for $q \neq 1$. 
Given $\bla = (\bla_0, \bla_1, \ldots, \bla_{r-1})$, an $r$-partite partition of $n$, recall that the notation $(\chi^{\bla})^*$ denotes an irreducible character of $\Grqn$ which appears in the decomposition of the restriction of the irreducible character $\chi^{\bla}$ of $G(r,1,n)$ to $\Grqn$. Note that $(\chi^{\bla})^*$ may not be unique and is an irreducible character $\chi^{(\tbla,\delta)}$, where $\tbla$ is an $(m,q)$-necklace obtained from $\bla$ and $\delta \in H_{\bla}$.

Proposition \ref{addcases} describes two specific cases for particular values of $r,q$ and $n$ important in the study of quasi $p$-Steinberg characters of $G(r,q,n)$ in Theorem \ref{MTqnot1}. For $n=3$ and $ 0 \leq j \neq k \neq l \leq r-1$, define $$\bnu^{j,k,l} = (\bnu_0, \bnu_1, \ldots, \bnu_{r-1}) \in \Y(r,3) \mbox { such that } \bnu_j=\bnu_k=\bnu_l=(1).$$ For $n=4$ and $ 0 \leq j \neq k \leq r-1$, define $\bnu^{j,k} = (\bnu_0, \bnu_1, \ldots, \bnu_{r-1}) \in \Y(r,4)$ such that $ \bnu_j=\bnu_k \vdash 2.$ For a multiple $r$ of $3$, define $$X_1 = \{(j,k,l) \mid 0 \leq j \leq r-1, k = (j+ \frac{r}{3}) \mkern-18mu \pmod{r}, l = (j+ \frac{2r}{3}) \mkern-18mu \pmod{r} \},$$
 and for $r$ even, define $X_2 = \{(j,k) \mid 0 \leq j \leq r-1,  k = (j+ \frac{r}{2}) \mkern-6mu \pmod{r} \}.$

\begin{prop}\phantomsection\label{addcases}
	\begin{enumerate}[(i)]
		\item The irreducible character $\chi^{\bnu^{j,k,l}}$ of $G(r,1,3)$ for $\bnu^{j,k,l} \in \Y(r,3)$  
		 decomposes into three two-dimensional irreducible characters of $G(r,q,3)$ if and only if $r, q$ are multiples of $3$ and $(j,k,l) \in X_1$. 
		
		\item The irreducible character $\chi^{\bnu^{j,k}}$ of $G(r,1,4)$ for $\bnu^{j,k} \in \Y(r,4)$ decomposes into two three-dimensional irreducible characters of $G(r,q,4)$ if and only if $r, q$ are even and $(j,k) \in X_2$. 
	\end{enumerate}
\end{prop}	

\begin{proof}
The proof follows by Theorem \ref{irrgp}.
\end{proof}	

We now prove Theorem \ref{MTqnot1} which characterizes quasi $p$-Steinberg characters of $G(r,q,n)$. In particular, this gives a classification of hefty quasi $p$-Steinberg characters of $G(r,q,n)$ in Theorem \ref{MTqnot1-hefty}.

	\begin{theorem}\label{MTqnot1}
		Given an $r$-partite partition $\bla = (\bla_0, \bla_1, \ldots, \bla_{r-1})$ of $n$, a non-linear irreducible character $(\chi^{\bla})^*$ of $G(r,q,n)$ is a quasi $p$-Steinberg character if and only if one of the following holds:
				\begin{enumerate}[(i)]
           \item $\bla$ corresponds to a quasi $p$-Steinberg character of $G(r,1,n)$ which are given in Table \ref{tab:quasi-Gr1n}. In this case, 
           $$ (\chi^{\bla})^* = \mbox{Res}^{G(r,1,n)}_{\Grqn}\chi^{\bla}.$$
           \item $n=3$, $r, q$ are multiples of $3$ and $\bla = {\bnu}^{j,k,l}$ for ${(j,k,l) \in X_1}$. 
           \item $n=4$, $r, q$ are even and $\bla = {\bnu}^{j,k}$ for $(j,k) \in X_2$.
		\end{enumerate}	
	\end{theorem}
	
	\begin{proof}
	For an $r$-partite partition $\bla = (\bla_0, \bla_1, \ldots, \bla_{r-1})$ of $n$, consider a quasi $p$-Steinberg character $(\chi^{\bla})^*$ of $G(r,q,n)$. By Corollary \ref{resqs}, either all the irreducible characters $\chi^{(\tbla,\delta)}$ appearing in $\mbox{Res}^{G(r,1,n)}_{\Grqn}\chi^{\bla}$ are quasi $p$-Steinberg or none of these are. Therefore, we can arbitrarily choose $(\chi^{\bla})^*$.

	We consider the cases $p \nmid n$ and $p \mid n$ separately. 
	
	\noindent \textbf{Case 1:} $p \nmid n$. Since $G(r,q,n)$ is a normal subgroup of $G(r,1,n)$, our method is to apply Lemma \ref{nosplit} by choosing a $p$-regular element in $G(r,1,n)$ whose conjugacy class does not split in $G(r,q,n)$. In the light of Theorem \ref{split}, we need to consider two subcases: $p \nmid n-1$ and $p \mid n-1$.
	
	\textbf{Case 1(a):} $p \nmid n-1$. The element $\alpha_1 = (0,\ldots, 0; (1,2, \ldots, n-1))$ is a $p$-regular element of $\G(r,q,n)$. Thus $(\chi^{\bla})^* (\alpha_1) \neq 0$, and by Lemma \ref{nosplit} we get that  $\chi^{\bla}(\alpha_1) \neq 0$. By Murnaghan--Nakayama rule, either $\bla=\widehat{\bla}^j$ for some $0 \leq j \leq r-1$ or $\bla=\widehat{\bla}^{j,k}$ for some $0 \leq j\neq k\leq r-1$.
	
	 However, we do not get any non-linear quasi $p$-Steinberg character of $G(r,q,n)$ in this subcase because of the following reason. When $\bla = \widehat{\bla}^j$ for some $0 \leq j \leq r-1$, Theorem \ref{irrgp} implies that $ (\chi^{\bla})^* = \mbox{Res}^{G(r,1,n)}_{\Grqn}\chi^{\bla}.$ 	
	When $\bla=\widehat{\bla}^{j,k}$ for some $0 \leq j\neq k\leq r-1$, for $n \geq 3$ we have $(\chi^{\bla})^* = \mbox{Res}^{G(r,1,n)}_{\Grqn}\chi^{\bla}$.	
	From Table \ref{tab:quasi-Gr1n}, we can see that there does not exist a non-linear quasi $p$-Steinberg character of $G(r,1,n)$ when $p\nmid n$ and $p \nmid n-1$.

	For $n=2$ and $\bla=\widehat{\bla}^{j,k}$, $\chi^{\bla}$ is a two-dimensional character which is not a quasi $p$-Steinberg character of $G(r,1,2)$ for $p \neq 2$.  When  $\chi^{\bla}$ decomposes, the two one-dimensional characters appearing in $\mbox{Res}^{G(r,1,2)}_{G(r,q,2)}\chi^{\bla}$ are quasi $p$-Steinberg characters for any prime $p$ that divides the order of $G(r,q,2)$.
	
	\textbf{Case 1(b):} $p \mid n-1$. Then, $p \nmid n-2$ and $\alpha_2 = (0,\ldots,0; (1,2, \ldots,n-2))$ is a $p$-regular element of $\Grqn$. Thus $(\chi^{\bla})^* (\alpha_2) \neq 0$, and by Lemma \ref{nosplit} we have  $\chi^{\bla}(\alpha_2) \neq 0$. Murnaghan--Nakayama rule for $G(r,1,n)$ implies that $\bla$ must be of one of the following forms:
	\begin{enumerate}[(a)]
		\item $\bla = \widehat{\bla}^j$;
		\item $\bla=\widehat{\bla}^{j,k}$;
		\item $\bla_j \vdash n-2$, $ \bla_k \vdash 2$ for some $k \neq j$, $\bla_l = \emptyset$ for all $l \notin \{j,k\}$;
		\item $\bla_j \vdash n-2$, $ \bla_k = (1)$ for some $k \neq j$, $ \bla_l = (1)$ for some $l \notin \{j,k\}$, and $ \bla_u = \emptyset$ for all $u \notin \{j,k,l\}$;
	\end{enumerate}
	When $\bla$ is of one of the forms (a)-(d) and for $n \geq 5$, by Theorem \ref{irrgp} we have $$ (\chi^{\bla})^* = \mbox{Res}^{G(r,1,n)}_{\Grqn}\chi^{\bla}.$$ Furthermore, for $\bla$ of the form (b), (c) or (d), $\chi^{\bla}$ vanishes on the $p$-regular element $\alpha = (0, \ldots, 0; (1,2,\ldots,n)) \in G(r,q,n)$ which implies that the character $(\chi^{\bla})^*$ is not a quasi $p$-Steinberg character of $\Grqn$. Now consider that $\bla$ is of the form (a), i.e., $\bla = \widehat{\bla}^j$. By Corollary \ref{MN cor}(ii) and by $(\chi^{\widehat{\bla}^j})^*$ being a quasi $p$-Steinberg character of $G(r,q,n)$, for a $p$-regular element $\sigma \in S_n$, we have  $$ \varkappa^{\bla_j}(\sigma) = \chi^{\widehat{\bla}^j} (0, \ldots, 0; \sigma) = (\chi^{\widehat{\bla}^j})^*  (0, \ldots, 0; \sigma) \neq 0.$$ This implies that $\varkappa^{\bla_j}$ is a quasi $p$-Steinberg character of $S_n$. Thus, $\chi^{\widehat{\bla}^j}$ is a quasi $p$-Steinberg character of $G(r,1,n)$  by Proposition \ref{QGS}.
	
	We now consider $n < 5$. Of course, $n \neq 2 $ as $p \mid n-1$. For $n=3, p=2$, when $\bla$ is of the form (a), (b) or (c), then it can be seen from Table \ref{tab:quasi-Gr1n} that $\chi^{\bla}$ is a quasi $2$-Steinberg character of $G(r,1,3)$. Thus, $(\chi^{\bla})^*$ is a quasi $2$-Steinberg character of $G(r,q,3)$ because $$ (\chi^{\bla})^* = \mbox{Res}^{G(r,1,3)}_{G(r,q,3)}\chi^{\bla}.$$ A similar argument holds true when $n=4, p=3$ and $\bla$ is of the form (a). From Table \ref{tab:quasi-Gr1n}, we see that $\bla$ cannot be of the form (b) or (d) for $n=4, p=3$. 
	
	When $\bla$ is of the form (d) for $n=3, p=2$, i.e. $\bla = \bnu^{j,k,l}$, then Theorem \ref{p23}(i) and Proposition \ref{addcases} imply that $(\chi^{\bla})^*$ occurring in $\mbox{Res}^{G(r,1,3)}_{G(r,q,3)}\chi^{\bla}$ is a quasi $2$-Steinberg character if and only if $r, q$ are multiples of $3$ and $(j,k,l) \in X_1$. The analogous arguments can be used to deduce that when $\bla$ is of the form (c) for $n=4, p=3$, i.e. $\bla = \bnu^{j,k}$, the character $(\chi^{\bla})^*$ occurring in $\mbox{Res}^{G(r,1,4)}_{G(r,q,4)}\chi^{\bla}$ is a quasi $2$-Steinberg character if and only if $r,q$ are even and $(j,k) \in X_2$.
	
	\noindent \textbf{Case 2:} $p \mid n$. Then, $p \nmid n-1$ and $\alpha_1 = (0,\ldots,0; (1,2, \ldots,n-1))$ is a $p$-regular element of $\Grqn$. Thus $(\chi^{\bla})^* (\alpha_1) \neq 0$, and by Lemma \ref{nosplit} we get that  $\chi^{\bla}(\alpha_1) \neq 0$. By Murnaghan--Nakayama rule for $G(r,1,n)$, either $\bla=\widehat{\bla}^j$ for some $0 \leq j \leq r-1$ or $\bla=\widehat{\bla}^{j,k}$ for some $0 \leq j\neq k\leq r-1$. Using similar arguments as given in Case 1(b), we deduce that if $\bla=\widehat{\bla}^j$ for some $0 \leq j \leq r-1$, then $\chi^{\widehat{\bla}^j}$ is a quasi $p$-Steinberg character of $G(r,1,n)$.
	
	For $\bla=\widehat{\bla}^{j,k}$ for some $0 \leq j\neq k\leq r-1$ and  for $n \geq 5$, the character $(\chi^{\bla})^*$ is not a quasi $p$-Steinberg character of $\Grqn$ because of the following observations:
	\begin{enumerate}[(a)]
			\item $\alpha_2 = (0,\ldots, 0; (1,2, \ldots, n-2))$ and $\alpha_3 = (0,\ldots, 0; (1,2, \ldots, n-3))$ are $p$-regular elements of $\Grqn$ for $p > 3$;
			\item $\alpha_2$ is $3$-regular and $\alpha_3$ is $2$-regular;
			\item $\chi^{\bla}(\alpha_2) = \chi^{\bla}(\alpha_3) = 0$; 
			\item $(\chi^{\bla})^* = \mbox{Res}^{G(r,1,n)}_{\Grqn}\chi^{\bla}.$
		\end{enumerate} 
	
Now, we discuss the case when $\bla=\widehat{\bla}^{j,k}$ for some $0 \leq j\neq k\leq r-1$ and  $n < 5$.	For $n=2, p=2$, we have $\La_j=(1)$ for some $j$, $\La_k = (1)$ for some $k \neq j$, $\La_l = \emptyset$ for all $l \notin \{j,k\}$. Therefore, the character $\chi^{\bla}$ is a two-dimensional quasi $2$-Steinberg character of $G(r,1,2)$. The character $(\chi^{\bla})^*$ is a non-linear quasi $2$-Steinberg character of $G(r,q,2)$ if and only if  $(\chi^{\bla})^* = \mbox{Res}^{G(r,1,2)}_{G(r,q,2)}\chi^{\bla}$. The character $\chi^{\bla}$ decomposes into two linear characters of $G(r,q,2)$ if and only if both $r$ and $q$ are even, and $k = j+\frac{r}{2}\mod(r)$. 

For $n=3, p=3$ and $n=4, p=2$, given $\bla=\widehat{\bla}^{j,k}$ implies that $(\chi^{\bla})^* = \mbox{Res}^{G(r,1,n)}_{G(r,q,n)}\chi^{\bla}$ (by Theorem \ref{irrgp}). Also, all the possible $\widehat{\bla}^{j,k}$ appear in Table \ref{tab:quasi-Gr1n}. Thus, for $n=3, p=3$ and $n=4, p=2$, the character $(\chi^{\widehat{\bla}^{j,k}})^*$ is a quasi $p$-Steinberg character of $G(r,q,n)$ if and only if $\chi^{\widehat{\bla}^{j,k}}$ is a quasi $p$-Steinberg character of $G(r,1,n)$.
\end{proof}	

\section*{Acknowledgements}
The authors thank Dipendra Prasad for encouragement regarding this project. The authors also thank Anupam Singh and Manoj Yadav for organizing Group Theory Sangam seminar series during which one of the seminars led to this project.

\bibliographystyle{plain}
\bibliography{Steinberg-CRG}
	
\end{document}